\documentclass[a4paper,10pt,reqno]{amsart}

\usepackage[utf8]{inputenc}
\usepackage{verbatim}
\usepackage{amsthm}
\usepackage{amsmath}
\usepackage{amssymb}
\usepackage{enumerate}
\usepackage{hyperref}
   \usepackage{geometry}
\numberwithin{equation}{section}

 \usepackage{tikz}
 \usetikzlibrary{matrix}

\usepackage{xcolor}

\usepackage{t1enc}

\newcommand{\mc}[1]{\mathcal{#1}}

\newtheorem{theorem}{Theorem}[section]
\newtheorem{proposition}{Proposition}[section]

\newtheorem{lemma}[theorem]{Lemma}
\newtheorem{problem}[theorem]{Problem}

\newtheorem{corollary}[theorem]{Corollary}

\newtheorem{fact}[theorem]{Fact}

\newtheorem{example}[theorem]{Example}

\numberwithin{theorem}{section}

\renewcommand{\em}{\bf}

\theoremstyle{definition}

\newtheorem{definition}[theorem]{Definition}

\newcommand{\setm}{\setminus}
\newcommand{\empt}{\emptyset}
\newcommand{\subs}{\subset}
\newcommand{\dom}{\operatorname{dom}}

\def\<{\left\langle}
\def\>{\right\rangle}

\newcommand{\we}{\operatorname{w}}
\newcommand{\den}{\operatorname{d}}
\newcommand{\cel}{\operatorname{c}}

\newcommand{\lin}{\operatorname{L}}
\newcommand{\ce}{\operatorname{cel}}
\newcommand{\sh}{\operatorname{sh}}

\def\br#1;#2;{\bigl[ {#1} \bigr]^ {#2} }

\newcommand{\ostwo}{\operatorname{os_2}}
\newcommand{\osthree}{\operatorname{os_3}}
\newcommand{\costwo}{\operatorname{cs_2}}
\newcommand{\costhree}{\operatorname{cs_3}}
\newcommand{\cosi}{\operatorname{cs_i}}
\newcommand{\osi}{\operatorname{os_i}}

\newcommand{\RO}{\operatorname{RO}}

\newcommand{\inte}{\operatorname{int}}

\author[I. Juh\'asz]{Istv\'an Juh\'asz}
\address
      { Alfréd Rényi Institute of Mathematics, Hungarian Academy of Sciences}
\email{juhasz@renyi.hu}

\author[L. Soukup]{Lajos Soukup}
\address
      { Alfréd Rényi Institute of Mathematics, Hungarian Academy of Sciences}
\email{soukup@renyi.hu}

\author[Z. Szentmikl\'ossy]{Zolt\'an Szentmikl\'ossy}
\address{E\"otv\"os University of Budapest}
\email{szentmiklossyz@gmail.com
}

\title[Small nowhere constant images]
{Spaces of small cellularity have nowhere constant continuous images  of small weight}

\date{\today}
\thanks{The preparation of this paper was
supported by  NKFIH grants K113047 and K129211.}

\begin{document}

\begin{abstract}
We call a continuous map $f : X \to Y$ {\em nowhere constant} if it is not constant on
any non-empty open subset of its domain $X$. Clearly, this is equivalent with
the assumption that every fiber $f^{-1}(y)$ of $f$ is nowhere dense in $X$. We call the
continuous map $f : X \to Y$ {\em pseudo-open} if for each nowhere dense $Z \subs Y$ its
inverse image $f^{-1}(Z)$ is nowhere dense in $X$. Clearly, if $Y$ is crowded, i.e. has
no isolated points, then $f$ is nowhere constant.

The aim of this paper is to study the following, admittedly imprecise, question:
How "small" nowhere constant, resp. pseudo-open continuous images can "large" spaces have?
Our main results yield the following two precise answers to this question, explaining
also our title. Both of them involve the cardinal function $\widehat{c}(X)$, the "hat version"
of cellularity, which is defined as the smallest cardinal $\kappa$ such that there is no
$\kappa$-sized disjoint family of open sets in $X$. Thus, for instance, $\widehat{c}(X) = \omega_1$
means that $X$ is CCC.

\smallskip

THEOREM A. Any crowded Tychonov space $X$ has a crowded Tychonov nowhere constant continuous image
$Y$ of weight $\we(Y) \le \widehat{c}(X)$. Moreover, in this statement $\le$ may be replaced with $<$
iff there are no $\widehat{c}(X)$-Suslin lines (or trees).

\smallskip

THEOREM B. Any crowded Tychonov space $X$ has a crowded Tychonov pseudo-open continuous image
$Y$ of weight $\we(Y) \le 2^{<\widehat{c}(X)}$. If Martin's axiom holds then there is a CCC crowded Tychonov space $X$ such that
for any crowded Hausdorff pseudo-open continuous image $Y$ of $X$ we have
$\we(Y) \ge \mathfrak{c}\,( = 2^{< \omega_1})$.

\end{abstract}

\dedicatory{Dedicated to Professor Arhangel'skii, on the occasion of his 80th birthday}

\maketitle

\section{Introduction}

In this paper all spaces are assumed to be crowded Hausdorff spaces. 
A space is called crowded if it has no isolated points.
Moreover,
all maps considered are continuous maps between such spaces.

Of course, all spaces with additional properties are also assumed to be crowded.
One, perhaps less widely known such property which we shall frequently assume is that of $\pi$-regularity.
A (Hausdorff) space $X$ is  $\pi$-regular if for every non-empty open set $U$ in $X$
there is a non-empty open set $V$ such that $\overline{V} \subs U$, 
in other words: the regular closed subsets of $X$ form a $\pi$-network in $X$, see e.g. \cite{JSSz}.

For any space $X$ we denote by $\tau(X)$ the topology of $X$ and put $\tau^+(X) = \tau(X) \setm \{\emptyset\}$.
Similarly, $RO(X)$ denotes the collection of all regular open sets in $X$, moreover $RO^+(X) = RO(X) \setm \{\emptyset\}$.
We shall denote by $CR(X)$ the collection of all non-empty
crowded subspaces of the space $X$.

As is mentioned in the abstract, the aim of this paper is to examine what can be said about the weight of
nowhere constant continuous images of spaces. Here is the self-explanatory definition of such maps.

\begin{definition}
 A continuous map $f:X\to Y$  is {\em nowhere constant (NWC)}\index{nwc}\index{nowhere constant}
iff it is not constant on any non-empty open sets, i.e.
  \begin{displaymath}
   \forall U\in \tau^+(X)\ |f[U]|>1.
  \end{displaymath}
\end{definition}

The following simple proposition yields alternative definitions of this concept.

\begin{proposition}\label{tm:nwc-equivalence}
 For a continuous map $f:X\to Y$
 the following statements are equivalent:
\begin{enumerate}[(a)]
 \item $f$ is NWC,
 \item $f[U]$ is crowded for each $U\in \tau^+_X$,
\item $f^{-1}(y)$ is nowhere dense in $Y$, i.e. $\inte(f^{-1}(y))=\empt$ for each $y\in Y$.
 \end{enumerate}
\end{proposition}

 \begin{proof}
 (a) $\Leftrightarrow$ (b): It is immediate that $f[X]$ is crowded if $f$ is NWC. Next,
 if $U\in \tau^+(X)$ then the restriction $f\upharpoonright U$ is NWC as well.

 \smallskip
 \noindent (a) $\Leftrightarrow$ (c) is obvious because $f^{-1}(y)$ is closed for each $y \in f[X]$.
 \end{proof}

This proposition leads us to two natural strengthenings of the notion of NWC maps
that we shall also consider.

\begin{definition}
A continuous map $f:X\to Y$ is
{\em pseudo-open (PO)}\index{pseudo-open}\index{PO} iff
for every nowhere dense subset $Z$ of $Y$ its preimage $f^{-1}(Z)$ is nowhere dense in $X$.
\end{definition}

As singletons in a crowded space are nowhere dense, every PO map into a crowded space is NWC.
The following observation is obvious, so we leave its proof to the reader.

\begin{proposition}\label{tm:po-equivalence}
 For a continuous map $f:X\to Y$
 the following three statements are equivalent:
\begin{enumerate}[(a)]
\item $f$ is PO,
\item the preimage of any dense open subset of $Y$ is dense (and open) in $X$,
\item for each $U\in \tau^+(X)$ we have $int(\overline{f[U]}) \ne \emptyset$.
 \end{enumerate}
\end{proposition}

We recall that a continuous map $f:X\to Y$ is called quasi-open (QO) if for each $U\in \tau^+(X)$
we have $int(f[U]) \ne \emptyset$, hence every QO map is PO. In some important cases the converse of this also holds.

\begin{fact}
 If the PO map $f:X\to Y$ is also closed and $X$ is $\pi$-regular, in particular if $X$ is compact, then $f$ is QO.
\end{fact}

Now we give the second strengthening of the notion of NWC maps. This is based on the obvious fact that any non-empty
open subspace of a crowded space is crowded.

\begin{definition}
 A continuous map $f:X\to Y$  is {\em crowdedness preserving (CP)} iff the image of any crowded
subspace  of $X$ is crowded.
\end{definition}

We again have alternative characterizations of this notion.

\begin{theorem}\label{tm:cp-equivalence}
 For any continuous map $f:X\to Y$
 the following statements are equivalent:
\begin{enumerate}[(a)]
 \item $f$ is CP,
 \item $|f[S]|>1$ for each  $S\in CR(X)$,
 \item the preimage of every point of $Y$ is scattered.
\end{enumerate}
\end{theorem}

\begin{proof}(a) $\Rightarrow$ (b): If $S\in CR(X)$ then $f[S]$ is crowded and so infinite.

 \smallskip \noindent (b) $\Rightarrow$ (c):
If $S\subs f^{-1}(y)$ for some $y\in Y$, then
$f[S]\subs \{y\}$, and so $S$ is not crowded by (b).
Thus $f^{-1}(y)$ does not contain any crowded subspaces, i.e.
$f^{-1}(y)$ is scattered.

  \smallskip \noindent (c) $\Rightarrow$ (a):
Assume (c) and that $f[S]$ is not crowded for some non-empty $S\subs X$.
Then there is an open $W\in \tau^+(Y)$ such that $f[S]\cap W=\{f(x)\}$
for some $x\in S$.
 Since $f$ is continuous, there is an open neighborhood $U$ of $x$
 such that $f[U]\subs W$, hence $f[U\cap S]=\{f(x)\}$. Thus $U\cap S$
 is scattered by (c), and so   $S$ is not crowded. Thus (a) holds.
\end{proof}

\section{Splitting families and splitting numbers}

In this section we introduce two kinds of splitting families that will turn out to play
an essential role in finding NWC or CP images of "small" weight of certain spaces.

\begin{definition}
Assume that $X$ is a space and   $\mc A,\mc B\subs \mc P(X)$.
 We say that
 \begin{enumerate}[(1)]
  \item $\mc A$ {\em $T_2$-splits} $\mc B$
  (or  $\mc A$ is a {\em$T_2$-splitting family} for $\mc B$) iff
  \begin{displaymath}
   \forall B\in \mathcal B \ \exists A_0, A_1\in \mathcal A\
   (A_0\cap A_1=\emptyset\land A_0\cap B\ne \emptyset\ne A_1\cap B).
  \end{displaymath}
 \item $\mc A$ {\em $T_3$-splits} $\mc B$
     (or  $\mc A$ is a {\em$T_3$-splitting family} for $\mc B$) iff
  \begin{displaymath}
   \forall B\in \mc B \ \exists A_0, A_1\in \mc A
   \ (\overline{A_0}\cap \overline{A_1}=\emptyset\land A_0\cap B\ne \emptyset\ne A_1\cap B).
  \end{displaymath}
 \end{enumerate}
\end{definition}

Actually, in all the interesting cases for us the splitting family $\mathcal{A}$ will consist of open sets;
this clearly explains and justifies our terminology. Also, our next result already shows
how to obtain NWC (resp. CP) Tychonov images of a normal space $X$ in the case that $\mathcal{B} = \tau^+(X)$
(resp. $\mathcal{B} = CR(X)$ ). 

\begin{theorem}\label{tm:nwc-from-splitting-normal}
Assume that  $X$ is a normal space and $\mc A \subs \mc P(X)$ is an infinite family that
$T_3$-splits $\mc B \subs \mc P(X)$. Then there is a continuous map
$f:X\to [0,1]^{|\mc A|}$ such that $|f[B]|>1$ for each $B\in \mc B$, i.e.
$f$ is not constant on any member of $\mathcal{B}$.
\end{theorem}

\begin{proof}
Let us put $$ \mc J=\big\{\<A_0,A_1\>\in \br \mc A;2;: \overline{A_0}\cap \overline{A_1}=\empt\big\}. $$
Since $X$ is normal,
for each $\<A_0,A_1\>\in \mc J$  we can choose a continuous function $f_{\<A_0,A_1\>}:X\to [0,1]$ such that
\begin{displaymath}
 f[A_0]=\{0\}\land f[A_1]=\{1\}.
\end{displaymath}
We then define $f:X\to [0,1]^{\mc J}$   by the formula
\begin{displaymath}
   f(x)(\<A_0,A_1\>)=f_{\<A_0,A_1\>}(x).
\end{displaymath}

Now, for every $B\in \mc B$ we can pick $\<A_0,A_1\>\in J$
such that $A_i\cap B\ne \empt$ for $i < 2$.  So, if we pick $x_i\in A_i\cap B$
then we have $$f(x_0)(\<A_0,A_1\>)=0 \text{ and } f(x_1)(\<A_0,A_1\>)=1\,,$$ hence
$f(x_0) \ne f(x_1)$. Since we have $|\mathcal{A}| = |\mathcal{J}|$, we are done.
\end{proof}

Let us point out that the generality of admitting arbitrary sets in the $T_3$-splitting family $\mathcal{A}$
is only apparent. Indeed, if $X$ is normal and $\mc A\,\, T_3$-splits $\mc B$ then we can trivially find
a family $\mc U \subs \tau(X)$ such that $|\mc U| = |\mc A|$ and $\mc U$ also $T_3$-splits $\mc B$.
The following, slightly less trivial, result takes this idea one step further: It shows how to obtain
a $T_3$-splitting family $\mathcal{U}$ in a fixed open base $\mc V$ of $X$. However, in this result we may
guarantee $|\mc U| = |\mc A|$ only if $|\mc A| \ge \lin(X)$.

 \begin{theorem}\label{tm:t3cutsfrombase}
Assume that $X$ is normal and the infinite family $\mc A$ $\,T_3$-splits $\mc B$,
moreover $\mc V$ is an open base of $X$.
Then there is $\mc U\in \br \mc V; \le |\mc A|\cdot \lin(X);$ such that $\mc U$ also $T_3$-splits $\mc B$.
\end{theorem}

\begin{proof}
Since $X$  is normal,
for each pair $A_0,A_1\in \mc A$ with $\overline{A_0}\cap \overline{A_1}=\empt$ we can choose
 subsets $\mc U(A_0,A_1,0)$
and $\mc U(A_0,A_1,1)$ of $\mc V$ with cardinality $\le \lin(X)$ such that
\begin{displaymath}
 \overline{A_i}\subs \bigcup \mc U(A_0,A_1,i)
\end{displaymath}
for $i<2$ and
\begin{displaymath}
\overline{ \bigcup \mc U(A_0,A_1,0)}\cap
\overline{ \bigcup \mc U(A_0,A_1,1)}=\empt.
\end{displaymath}
Then
\begin{displaymath}
 \mc U=\bigcup\{\mc U(A_0,A_1,i): A_0, A_1\in \mc A, \overline{A_0}\cap \overline{A_1}=\empt, i=0,1\}
\end{displaymath}
$T_3$-splits $\mc B$.
Indeed, given any $B\in \mc B$ there is $\{A_0, A_1\}\in \br \mc A;2;$
such that $A_0\cap B\ne \empt\ne A_1\cap B$ and
$\overline {A_0}\cap \overline{A_1}=\empt$.
Pick $x_i\in B\cap A_i$. Then there is $U_i\in U(A_0,A_1,i)$
with $x_i\in U_i$ for $i<2$.
Then $\overline{U_0}\cap \overline{U_1}=\empt$
and $U_i\cap B\ne \empt$. This clearly shows that $\mc U$
$T_3$-splits $\mc B$, moreover it is obvious that $|\mc U| \le |\mc A|\cdot \lin(X)$.
\end{proof}

We also have the following somewhat analogous result for $T_2$-splitting.

\begin{theorem}\label{tm:t2cutsfrombase}
Assume that $\mc V$ is a $\pi$-base of  the space $X$ and $\mc A \subs \tau(X)$ $\,T_2$-splits some $\mc B \subs \tau^+(X)$.
Then there is $\mc U \subs \mc V$ with $|\mc U| \le |\mc A| \cdot \widehat{c}(X)$ such that $\mc U$ also $T_2$-splits $\mc B$.
Moreover, if $|\mc A| < \widehat{c}(X)$ then we may even have $|\mc U| < \widehat{c}(X)$.
\end{theorem}

\begin{proof}
Let us fix for each open set $A \in \mc A$ a disjoint family $\mc U_A \subs \mc V$ such that
its union is dense in $A$. This is possible because $\mc V$ is a $\pi$-base of $X$.
Then for any $A \in \mc A$ and $B \in \mc B$ we have $A \cap B \ne \emptyset$ iff $\cup \mc U_A \cap B \ne \emptyset$,
hence $$\mc U = \bigcup \{\mc U_A : A \in \mc A \}$$
is as required. Moreover, as $\widehat{c}(X)$ is always regular, $|\mc A| < \widehat{c}(X)$ implies that $|\mc U| < \widehat{c}(X)$.
\end{proof}
Note that in this result it is essential that the members of both $\mc A$ and $\mc B$ are open.

It is immediate from Theorem \ref{tm:nwc-from-splitting-normal} that if $X$ is normal and $\mc A$ $\,T_3$-splits
$\tau^+(X)$ (resp. $CR(X)$) then $X$ admits an NWC (resp. CP) Tychonov image
of weight $\le |\mc A|$.
This justifies the introduction and study of the following cardinal functions
that are naturally called open (resp. crowded) splitting numbers.

To start with, let us point out that any space (i.e. crowded $T_2$ space) $X$ admits a family of open
sets that $T_2$-splits $\tau^+(X)$ (resp. $CR(X)$). If, in addition, $X$ is also Urysohn or $\pi$-regular then
$X$ admits a family open sets that $T_3$-splits $\tau^+(X)$. Any Urysohn-space $X$ also admits  open families that $T_3$-split $CR(X)$.

\begin{definition} For any space $X$  we let
 \begin{displaymath}
 \ostwo(X)=\min\{|\mc U|: \mc U\subs \tau(X) \text{ and } \mc U \text{ $T_2$-splits $\tau^+(X)$}\},
 \end{displaymath}
 and
 \begin{displaymath}
 \costwo(X)=\min\{|\mc U|: \mc U\subs \tau_X \text{ and } \mc U \text{ $T_2$-splits  $CR(X)$}\}.
 \end{displaymath}
 \end{definition}

\begin{definition} For any space $X$ that admits a family of open sets which $T_3$-splits $\tau^+(X)$ (resp. $CR(X)$) we let
 \begin{displaymath}
 \osthree(X)=\min\{|\mc U|: \mc U\subs \tau_X \text{ and } \mc U \text{ $T_3$-splits $\tau^+(X)$}\},
 \end{displaymath}
 and
 \begin{displaymath}
 \costhree(X)=\min\{|\mc U|: \mc U\subs \tau_X \text{ and } \mc U \text{ $T_2$-splits  $CR(X)$}\}.
 \end{displaymath}
 \end{definition}

Let us remind the reader here that our main interest lies in the study of the weight of NWC (as well as CP and PO)
images of spaces. Now, it turns out that the splitting numbers we have just defined yield lower bounds for these.

\begin{theorem}\label{lm:os2mwcU}
  If  $Y$ is any NWC image of the space $X$ then
  \begin{displaymath}
     \ostwo(X)\le\we(Y).
  \end{displaymath}
Moreover, if  $Y$ is Urysohn then we even have
 \begin{displaymath}
     \osthree(X)\le\we(Y).
  \end{displaymath}
 \end{theorem}

\begin{proof}
Let $f$ be an NWC map of $X$ onto $Y$ and $\mc B$ be an open base of $Y$.
Then  $\mc C=\{f^{-1}U:U\in \mc B\}$ $T_2$-splits $\tau^+(X)$.
 Indeed, since $f$ is NWC, for any $G\in \tau^+(X)$ we can pick points $x,y\in G$
 such that $f(x)\ne f(y).$  Now, if $U,\,V \in \mc B$   are disjoint neighborhoods of $f(x)$
 and $f(y)$, respectively, then $f^{-1}U$ and $f^{-1}V$ are disjoint members of  $\mc C$
 and both intersect $G$.

If $Y$ is Urysohn then  $U,\,V \in \mc B$ can be chosen so that, in addition,
$\overline U\cap \overline V=\empt.$ But then, since $f$ is continuous,
$f^{-1}U$ and $f^{-1}V$ are neighborhoods of $x$
and $y$ with disjoint closures, both intersecting $G$. Consequently,  $\mc C$ $\,T_3$-splits $\tau^+(X)$.
 \end{proof}

Practically the same argument we just gave yields the following analogous result for the weight of CP images,
hence we omit its proof.

\begin{theorem}\label{lm:cosilew}
If  $Y$ is any CP image of the space $X$ then
  \begin{displaymath}
     \costwo(X)\le\we(Y).
  \end{displaymath}
Moreover, if  $Y$ is Urysohn then we even have
 \begin{displaymath}
     \costhree(X)\le\we(Y).
  \end{displaymath}
\end{theorem}

The last two results together with Theorem \ref{tm:nwc-from-splitting-normal} immediately imply the
following corollaries. The second corollary uses the fact that, by Theorem \ref{tm:cp-equivalence},
if the map $f : X \to Y$ is not constant on any member of $CR(X)$ then $f$ is CP.

\begin{corollary}\label{tm:os2nwcnormal}
For every normal space $X$ we have
 \begin{align}\notag
  \osthree(X)=&\min\{\we(Y): Y\text{ is a Urysohn NWC image of } X\}\\\notag
=&\min\{\we(Y): Y\text{ is a Tychonov NWC image of } X\}.
 \end{align}
\end{corollary}

\begin{corollary}\label{tm:cos2nwcnormal}
For every normal space $X$ we have
 \begin{align}\notag
  \costhree(X)=&\min\{\we(Y): Y\text{ is a Urysohn CP image of } X\}\\\notag
=&\min\{\we(Y): Y\text{ is a Tychonov CP image of } X\}.
 \end{align}
\end{corollary}

We do not know if these results remain valid after weakening 'normal' to 'Tychonov'.
In particular, the following questions are open.

\begin{problem}
 Is there a  Tychonov space $X$ such that
 \begin{displaymath}
  \osthree(X)<
\min \{\we(Y): Y\text{ is a Tychonov NWC image of } X\} ?
 \end{displaymath}
 or
\begin{displaymath}
  \osthree(X)<
\min\{\we(Y):  Y\text{ is a Urysohn NWC image of } X\} ?
 \end{displaymath}
 \end{problem}

We do not have an analogue of Theorem \ref{tm:nwc-from-splitting-normal} for PO maps,
however the open splitting numbers $ \osi(X)$ turn out to be lower bounds even for the
$\pi$-weights of appropriate PO images.

\begin{lemma}\label{lm:os2pot2}
 If  $Y$ is any PO image of the space $X$ then
  \begin{displaymath}
     \ostwo(X)\le\pi(Y).
  \end{displaymath}
Moreover, if  $Y$ is $\pi$-regular then we even have
 \begin{displaymath}
     \osthree(X)\le\pi(Y).
  \end{displaymath}
\end{lemma}

\begin{proof}
If $\mc B$ is any $\pi$-base of $Y$ then  $\mc C=\{f^{-1}U:U\in \mc B\}$ will $T_2$-split $\tau^+(X)$.
Indeed, let $G\in \tau^+_X$. Since $f$ is PO, then $W = int(\overline{f[G]}) \in \tau^+(Y)$, hence, as $Y$ is crowded, we can pick disjoint elements
$U$  and $V$ of $\mc B$ such that $U \cup V \subs W$. But then $U \cap f[G] \ne \emptyset \ne V \cap f[G]$,
Consequently,  $f^{-1}U$ and $f^{-1}V$  $\,T_2$-split $G$.

If $Y$ is  Urysohn then we may first fix two open sets $U_0$ and $V_0$ with disjoint closures both intersecting $W$.
Then pick $U, V \in \mc B$ such that $U \subs W \cap U_0$ and $V \subs W \cap V_0$. Clearly, then
$f^{-1}U$ and $f^{-1}V$ will $\,T_3$-split $G$.
\end{proof}

\medskip

We have the following trivial relationships between the values of the splitting numbers,
whenever they are defined.

\begin{center}
 \begin{tikzpicture}
  \matrix (m) [matrix of math nodes,row sep=3em,column sep=4em,minimum width=2em]
  {
     {\costwo(X)} & {\costhree(X)}  \\
   {\ostwo(X)}   & {\osthree(X)}   \\};

\path [font=\scriptsize,draw=white] (m-1-1) edge
 [bend right=0] node {$\le$} (m-1-2);
\path [font=\scriptsize,draw=white] (m-2-1) edge
 [bend right=0] node[rotate=90] {$\le$} (m-1-1);
\path [font=\scriptsize,draw=white] (m-2-1) edge
 [bend right=0] node {$\le$} (m-2-2);
\path [font=\scriptsize,draw=white] (m-2-2) edge
 [bend right=0] node[rotate=90] {$\le$} (m-1-2);
 \end{tikzpicture}

\end{center}

In \ref{ex:easy} below we shall show that no other relation holds between these splitting numbers,
even for compact spaces. To this end we need some preparation.

\begin{theorem}\label{tm:prod-osi}
If $X$ and $Y$ are appropriate spaces then
 $$\osi(X\times Y)\le \min(\osi(X),\osi(Y))$$ for  $i=2,3$.
If both $X$ and $Y$ are regular Lindelöf spaces then
$$\osthree(X\times Y)=\min(\osthree(X),\osthree(Y)).$$
\end{theorem}

\begin{proof}
 If $\mc A$  $T_i$-splits $\tau^+(X)$ then
 $\{A\times Y: A\in \mc A\}$  $T_i$-splits $\tau^+(X\times Y)$. By symmetry, this clearly implies
$\osi(X\times Y)\le \min(\osi(X),\osi(Y)$.

Assume now that both $X$ and $Y$ are regular Lindelöf and 
$\mc A \subs \tau(X\times Y)$ is chosen so that $\mc A$ $\,T_3$-splits $\tau^+(X\times Y)$ and $|\mc A| = \osthree(X\times Y)$.
Since regular Lindelöf spaces are normal, we may apply
Theorem \ref{tm:t3cutsfrombase}  to assume that every $A\in \mc A$ has the form $A = U_A\times V_A$ with $U_A \in \tau(X)$ and $V_A \in \tau(Y)$.
We claim that either $\{U_A:A\in \mc A\}$ $\,T_3$-splits $\tau^+(X)$,
or $\{V_A:A\in \mc A\}$ $\,T_3$-splits $\tau ^+(Y)$.
Indeed, assume that $\{U_A:A\in \mc A\}$ does not $T_3$-split some $U \in \tau^+(X)$.
Since for every $V \in \tau^+(Y)$ there are $A, B \in \mc A$ which $T_3$-split $U \times V$,
then we must have that $V_A$ and $V_B$ $\,T_3$-split $V$. Consequently $\{V_A:A\in \mc A\}$ $\,T_3$-splits $\tau ^+(Y)$.
But then $\osthree(X\times Y)\ge \min(\osthree(X),\osthree(Y))$, completing the proof.
\end{proof}

Interestingly, the behavior on products of the crowded splitting numbers is quite different.

\begin{theorem}\label{tm:prod-cosi}
If $X$ and $Y$ are appropriate spaces then
 $$\cosi(X\times Y)=\max(\cosi(X),\cosi(Y))$$ for  $i=2,3$.
\end{theorem}

\begin{proof}
If $E\subs X\times Y$ is crowded then its projections $\pi_X(E)$ and $\pi_Y(E)$ cannot both be scattered.
Consequently, if $\mc U$ $T_i$-splits $CR(X)$ and $\mc V$  $T_i$-splits  $CR(Y)$
then $$\{U\times Y,\, X\times V : U\in \mc U,\,V\in \mc V\}$$ will $T_i$-split $CR(X\times Y).$
This implies $\cosi(X\times Y)\le \max(\cosi(X),\cosi(Y))$.

 On the other hand,  if $\mc W$ $\,T_i$-splits $CR(X\times Y)$ then, in particular,
$$\{W\cap (X\times \{y\}): W \in \mc W\}$$ clearly $T_i$-splits $CR(X\times \{y\})$ for every point $y \in Y$. Hence
$\cosi(X)\le \cosi(X\times Y)$, and by symmetry, we are done.
\end{proof}

Now, we are ready to present the promised compact examples that separate the corresponding values of our splitting numbers.
As usual, $D({\kappa})$ denotes the discrete space of size $\kappa$ and $\alpha(E)$ denotes the one-point
compactification of the locally compact space $E$.

\begin{example}\label{ex:easy}
For any fixed cardinal ${\kappa}> {\omega}$ consider the following compact spaces:
\begin{displaymath}
X=D(2)^{\kappa},\ Y= {\alpha}([0,1]\times D({\kappa})),
\ Z=X \oplus Y\text{ and }\ T=Y\times [0,1] .
\end{displaymath}
Then
\begin{enumerate}[(i)]
  \item $\ostwo(X)=\osthree(X)={\omega}<{\kappa}=
  \costwo(X)=\costhree(X)$,
  \item $\ostwo(Y)=\costwo(Y)={\omega}<{\kappa}=\osthree(Y)=\costhree(Y)$,
  \item $\ostwo(Z)={\omega}<{\kappa}=
  \costwo(Z)=\osthree(Z)=\costhree(Z)$,
  \item $\ostwo(T)=\costwo(T)=\osthree(T)={\omega}<{\kappa}=\costhree(T)$.
\end{enumerate}
\end{example}
\begin{proof}
(i)
If $\mc U \subs \tau^+(X)$ $\,T_2$-splits
some family $\mc B$ then clearly so does
\begin{displaymath}
  \mc V=\{\inte(\overline{U}):U\in \mc U\} \subs RO(X),
 \end{displaymath}
and clearly $|\mc V| \le |\mc U|$.
But any regular open set in $D(2)^{\kappa}$ depends only on
countably many coordinates. Hence any family of regular open
sets of cardinality less than ${\kappa}$ depends only on a set of coordinates $J$ with $|J| < \kappa$ and so it does not split e.g.
the crowded set $\{x \in D(2)^{\kappa} : \forall\,\alpha \in J\,(x(\alpha) = 0)\}$.  Thus we have  $\costwo(X)= \costhree(X) ={\kappa}$.

On the other hand, the countable family
\begin{displaymath}
  \big\{\{x\in D(2)^{\kappa}: x(n)=i\} : n<{\omega}, i<2\big\}
 \end{displaymath}
of clopen sets clearly $T_3$-splits $\tau^+(X)$ and thus we have $\ostwo(X) = \osthree(X) = {\omega}$.

\smallskip

\noindent (ii)
Clearly, if $S \subs Y$ is crowded then so is $S \cap ([0,1] \times \{\alpha\})$ for some $\alpha \in \kappa$.
Hence if $\{U_n:n\in {\omega}\}$ is any countable base of $[0,1]$ then $\{U_n\times D({\kappa}):i<{\omega}\}$ is a
countable $T_2$-splitting family for $CR(Y)$. Thus we have $\ostwo(Y)=\costwo(Y)={\omega}$.

Now, assume that $\,\mc U \subs \tau^+(Y)$ $\,T_3$-splits $\tau^+(Y)$ and
consider the family
$$\mc V=\{U\in \mc U: \exists I_U\in \br {\kappa};<{\omega};
\,(U\subs [0,1]\times I_U)\}.$$
We claim that $|\mc V| \ge {\kappa}$.
Indeed, otherwise for  $J=\bigcup\{I_U:I\in \mc V\}$ we had $|J|<{\kappa}$.
But then $W=[0,1]\times ({\kappa}\setm J) \in \tau^+(Y)$ and for any $U \in \mc U$ with $U\cap W\ne \empt$ we have $U \notin \mc V$,
hence the point at infinity of $Y$ belongs to $\overline{U}$. This, however would contradict our assumption that
$\,\mc U \subs \tau^+(Y)$ $\,T_3$-splits $\tau^+(Y)$.
Consequently, we have $|\mc U| \ge |\mc V| \ge \kappa$, and hence $\osthree(Y)=\costhree(Y) = \kappa$.

\smallskip

\noindent(iii) This is straightforward.

\smallskip

\noindent(iv)
By Theorem \ref{tm:prod-osi} we have
$\omega \le \ostwo(T)\le\osthree(T)\le \osthree([0,1])={\omega}$.
Next, Theorem \ref{tm:prod-cosi} implies both $\costwo(T)=\max(\costwo(Y),\costwo([0,1]))={\omega}$ and $\costhree(T)=\max(\costhree(Y),
\costhree([0,1]))={\kappa}$.
\end{proof}

\section{Shattering, splitting, and cellularity}

The aim of this section is to present the first of our two main results, Theorem A
from the abstract. The crucial step will be achieved by establishing that the cellularity number $\widehat{c}(X)$
is an upper bound for the splitting number $\ostwo(X)$ (resp. $\osthree(X)$ ) for all (resp. all $\pi$-regular) spaces.
This, in turn, will make use of the concepts of {\bf shattering family} and {\bf shattering number}
that we shall define below.

\begin{definition}\label{def:sh}
(i) For any space $X$ let
$$\ce(X) = \{\mc S \subs \RO^+(X) : \mc S \text{ is disjoint} \}.$$
(ii) $\mathfrak{F} \subs \ce(X)$ is called a shattering family for $X$
if for for any $U \in \tau^+(X)$ there is $\mc S \in \mathfrak{F}$ such that
$U$ intersects at least two members of $\mc S$, moreover
$$\sh(X) = \min \{|\mathfrak{F}| : \mathfrak{F} \text { is a shattering family for } X  \}$$
is called the shattering number of $X$.
\end{definition}

Every space $X$ (being crowded $T_2$) admits a shattering family.
Indeed, if  $\mc U \subs \RO^+(X)$ is a $T_2$-splitting family for $X$ then the family of pairs
$$\{\,\{U,\,X \setm \overline{U}\} : U \in \mc U \}$$ is a shattering family for $X$.
This, of course, also shows that $\ostwo(X) \ge \sh(X)$. On the other hand,
it is obvious that if $\mathfrak{F}$ is a shattering family for $X$ then $\cup \mathfrak{F}$
$\,T_2$-splits $\tau^+(X)$, hence $\ostwo(X) \le \sh(X) \cdot \cel(X)$.
If $X$ is also $\pi$-regular then
for any $R \in \cup \mathfrak{F}$ we may fix a collection $\mc S_R \subs \RO^+(X)$
such that for each $U \in \mc S_R$ we have $\overline{U} \subs R$, we have $\overline{U} \cap \overline{V} = \emptyset$  
for distinct $U,\,V \in \mc S_R$, and $\cup \mc S_R$ is dense in $R$. Now, if we replace any $\mc S \in \mathfrak{F}$
with $\widetilde{\mc S} = \bigcup \{\mc S_R : R \in \mc S\}$  then the union of $\widetilde{\mathfrak{F}} = \{\widetilde{\mc S} : \mc S \in \mathfrak{F}\}$
even $\,T_3$-splits $\tau^+(X)$, hence in this case we have  $\osthree(X) \le \sh(X) \cdot \cel(X)$.

Actually, for $\ostwo(X)$ the following stronger result holds. Recall that for any cardinal $\kappa$ its logarithm
is defined by $\log \kappa = \min \{\lambda : 2^\lambda \ge \kappa \}.$

\begin{theorem}\label{tm:log}
For any space $X$ we have $\ostwo(X) \le \sh(X) \cdot \log( \cel(X)).$
\end{theorem}

\begin{proof}
Let $\{\mc S_\alpha : \alpha < \sh (X)\}$ be a shattering family for $X$ and note that for
each $\alpha < \sh (X)\}$ we have $|\mc S_\alpha| \le \kappa = \cel(X).$ We may assume without
any loss of generality that each $\mc S_\alpha$ is infinite.

Now, it is well-known
that for every infinite set $S$ we have $\log |S|$ many binary partitions of $S$, say $\mc P$, such that for any two distinct
elements $x\,,y$ of $S$ there is some $P = \{P_0,P_1\} \in \mc P$ for which $x \in P_0$ and $y \in P_1$.
This follows e.g. from the fact that $\den (D(2)^\kappa) = \log \kappa$ for any $\kappa \ge \omega$, see e.g. \cite{Ju}.

So, let us fix for each $\alpha < \sh (X)$ such a system $\mc P_\alpha$ of binary partitions of $\mc S_\alpha$ with
$|\mc P_\alpha| = \log |\mc S_\alpha| \le \log \cel(X)$.  
Clearly, then $$\mc U = \{\cup P_i : P \in \cup_{\alpha < \sh (X)}\,\mc P_\alpha,\,\,i < 2\} \subs \tau(X)$$
$T_2$-separates $\,\tau^+(X)$ and $|\mc U| \le \sh(X) \cdot \log( \cel(X)).$
\end{proof}

We now present the crucial result of this section.

\begin{theorem}\label{tm:sh_le}
We have $\sh(X) \le \widehat{c}(X)$ for every space $X$. Moreover, if $\sh(X) = \widehat{c}(X)$ then there is
a $\widehat{c}(X)$-Suslin tree (or line).
\end{theorem}

\begin{proof}
We are going to build a tree $\mc T \subs \RO^+(X)$, with inclusion $\subs$ as the tree ordering,
of height $\eta \le \widehat{c}(X)$ whose levels
$\{\mc T_\alpha : \alpha < \eta\}$ will form a shattering family for $X$.
It is natural call this a shattering tree for $X$.

To begin with, for each regular open set $R \in \RO^+(X)$ we fix two disjoint non-empty regular open subsets of $R$, say $P(R)$ and $Q(R)$,
in such a way that their union $P(R) \cup Q(R)$ is dense in $R$. This is possible because
$X$ is crowded $T_2$. Note that then for every open subset $U$ of $R$ we have either $U \subs P(R)$, or $U \subs Q(R)$,
or $U$ intersects both $P(R)$ and $Q(R)$.

Now, we define the levels $\mc T_\alpha$ of $\mc T$ by transfinite recursion on $\alpha$. For a  start, we put $\mc T_0 = \{X\}$.
In the successor step, if $\mc T_\alpha$ is defined then we put $$\mc T_{\alpha + 1} = \{P(R) \,, Q(R) : R \in \mc T_\alpha \}.$$
This means that each member of $\mc T_\alpha$ has two immediate successors, in particular,
$\mc T_\alpha \ne \emptyset$ implies $\mc T_{\alpha+1} \ne \emptyset$.

If $\alpha$ is limit and $\mc T \upharpoonright \alpha = \cup_{\beta < \alpha} \mc T_\beta$ has been defined
then we first consider the set of all cofinal branches $B_\alpha$ of the partial tree $\mc T \upharpoonright \alpha$.
Formally, $B_\alpha$ consists of all functions $b : \alpha \to \mc T \upharpoonright \alpha$ such that $b(\beta) \in \mc T_\beta$ for any $\beta < \alpha$, moreover
$b(\beta) \supset b(\gamma)$ whenever $\beta < \gamma < \alpha$.
For every $b \in B_\alpha$ we let $R_b =  int(\bigcap_{\beta < \alpha}\overline{b(\beta)})$, then $R_b \ne \emptyset$ implies $R_b \in \RO^+(X)$.
Then we put $$\mc T_\alpha = \{R_b : b \in B_\alpha\} \setm \{\emptyset\}.$$

This recursive construction stops at the first, necessarily limit, ordinal $\eta$ when  $R_b = \emptyset$ for all $b \in B_\eta$.
It is clear from our construction that if $S$ and $T$ are incomparable elements of $\mc T = \cup_{\alpha < \eta} \{\mc T_\alpha : \alpha < \eta\}$
then $S \cap T = \emptyset$, hence every antichain in $\mc T$ has size $< \widehat{c}(X)$. Moreover, as $\mc T$ branches at
all of its nodes, every chain in $\mc T$ has size $< \widehat{c}(X)$ as well. It follows then that $\eta \le \widehat{c}(X)$,
moreover if $\eta = \widehat{c}(X)$ then $\mc T$ is a $\widehat{c}(X)$-Suslin tree.

Now, it remains to show that the family of levels
$\{\mc T_\alpha : \alpha < \eta\}$ of $\mc T$ forms a shattering family for $X$. To see this,
we shall prove by transfinite induction on $\alpha \le \eta$ the following statement $(I_\alpha)$:
For every $U \in \tau^+(X)$ either there is $R \in \mc T_\alpha$ with $U \subs R$, or there is a
$\beta \le \alpha$ such that $U$ intersects at least two members of $\mc T_\beta$.

Indeed, $(I_0)$ holds trivially. If $(I_\alpha)$ holds then again it is trivial that so does $(I_{\alpha+1})$.
So, assume that  $\alpha \le \eta$ is limit and $(I_\beta)$ holds for all $\beta < \alpha$ and fix  $U \in \tau^+(X)$.
If there is $\beta < \alpha$ such that $U$ intersects at least two members of $\mc T_\beta$ then we are done.
Otherwise, by the inductive assumption, for each $\beta < \alpha$ there is some $b(\beta) \in \mc T_\beta$
with $U \subs b(\beta).$ But then $b \in B_\alpha$ and we clearly have $U \subs \bigcap_{\beta < \alpha}\overline{b(\beta)}$,
hence $U \subs R_b \in \mc T_\alpha$. This shows that $(I_\alpha)$ is valid. Since $\mc T_\eta = \emptyset$, then $(I_\eta)$
simply means that $\{\mc T_\alpha : \alpha < \eta\}$ is indeed a shattering family for $X$.
Consequently we have $\sh(X) \le \eta \le \widehat{c}(X)$, moreover $\sh(X) = \widehat{c}(X)$ implies $\eta = \widehat{c}(X)$,
hence $\mc T$ is a $\widehat{c}(X)$-Suslin tree.
\end{proof}

Since we always have $\cel(X) \le \widehat{c}(X)$, from Theorem \ref{tm:sh_le} and the remarks made before it we immediately obtain the following results.

\begin{corollary}\label{co:os2os3}
For every space $X$ we have $\ostwo(X) \le \widehat{c}(X)$ and if $X$ is also $\pi$-regular then even
$\osthree(X) \le \widehat{c}(X)$. Moreover, if $\sh(X) < \widehat{c}(X)$ then $\ostwo(X) < \widehat{c}(X)$
and if $X$ is $\pi$-regular then $\osthree(X) < \widehat{c}(X)$.
\end{corollary}
The second part is immediate from the fact that $\widehat{c}(X)$ is always a regular cardinal.

We are now ready to present our main result about Tychonov NWC images of Tychonov spaces.

\begin{theorem}\label{tm:Tych}
Any Tychonov space $X$ has a Tychonov NWC image $Y$ of weight $\we(Y) \le \widehat{c}(X)$.
If $\sh(X) < \widehat{c}(X)$, in particular if there are no $\widehat{c}(X)$-Suslin trees,
then in this statement $\le$ can be replaced by $<$.
\end{theorem}

\begin{proof}
We shall actually give two proofs, one using $\ostwo$ and the other using $\osthree$.

{\bf First proof}. We may assume that $X$ is a subspace of some Tychonov cube $[0,1]^\kappa$ and will
show that there is $J \subs \kappa$ with $|J| \le \widehat{c}(X)$ such that the projection map
$\pi_J : [0,1]^\kappa \to [0,1]^J$ restricted to $X$ is NWC.

To see this, let us fix an open base $\mc B$ of $[0,1]$ and let $\mc E$ be the base of $[0,1]^\kappa$
consisting of all the elementary open sets of the form $[\varepsilon]$, where $\varepsilon$ is
a function with domain a finite subset of $\kappa$ and range included in $\mc B$, moreover
$$[\varepsilon] = \{x \in [0,1]^\kappa : \,\forall\,\alpha \in \dom(\varepsilon)\,\,\big(x(\alpha) \in \varepsilon(\alpha)\big\}.$$

Then $\mc V = \{[\varepsilon] \cap X : \varepsilon \in \mc E\}$ is a base for $X$, hence we may apply Theorem
\ref{tm:t2cutsfrombase} to obtain $\mc U \subs \mc V$ of size $\le \ostwo(X) \cdot \cel(X) \le \widehat{c}(X)$
which $\,T_2$-splits $\tau^+(X)$. Let us fix for each $U \in \mc U$ an $\varepsilon_U \in \mc E$
such that $U = [\varepsilon_U] \cap X$ and set $J = \bigcup \{\dom(\varepsilon_U) : U \in \mc U\}$,
then $|J| \le |\mc U| \le \widehat{c}(X)$.

To see that $\pi_J$ is NWC on $X$, pick any $G \in \tau^+(X)$ and two disjoint members  $U_0$ and $U_1$
of $\mc U$ both intersecting $G$. But then the non-empty projections $\pi_J\big[[\varepsilon_{U_0}] \cap G \big]$ and $\pi_J\big[[\varepsilon_{U_1}1] \cap G \big]$
are also disjoint because $\dom(\varepsilon_{U_0}) \cup \dom(\varepsilon_{U_1}) \subs J$.
Consequently, $|\pi_J[G]| > 1$.

If $\sh(X) < \widehat{c}(X)$ then Corollary \ref{co:os2os3} implies $\ostwo(X) < \widehat{c}(X)$,
hence, by  Theorem \ref{tm:t2cutsfrombase} again, the $\,T_2$-splitting family $\mc U$ can be found with $|\mc U| < \widehat{c}(X)$,
and then we have $|J| < \widehat{c}(X)$ as well.

{\bf Second proof}. Let $bX$ be any compactification of $X$, then $bX$ is normal and hence \ref{tm:os2nwcnormal} implies that there is an NWC surjection $f : bX \to Y$,
where $Y$ is (necessarily) compact with $$\we(Y) = \osthree(bX) \le \widehat{c}(bX) = \widehat{c}(X),$$
using Corollary \ref{co:os2os3}.
But, as $X$ is dense in $bX$, then $f \upharpoonright X$ is also NWC. Indeed, this is because for every $U \in \tau^+(bX)$ we have
$x, y \in U$ with $f(x) \ne f(y)$ and so disjoint open sets $V,\, W$ in $Y$ containing $f(x)$, resp. $f(y)$. Then
$x \in f^{-1}(V) \cap U$ and $y \in f^{-1}(W) \cap U$ imply that we can pick
points $x' \in X \cap f^{-1}(V) \cap U$ and $y' \in X \cap f^{-1}(W) \cap U$, concluding that $f(x') \ne f(y')$ for the points $x',y' \in X \cap U$.
Now, note that $\we(f[X]) \le \we(Y)$.

Since $X$ is dense in $bX$, we clearly have $\sh(X) = \sh(bX)$. Hence $\sh(X) < \widehat{c}(X)$ implies that, by Corollary \ref{co:os2os3} again,
$$\we(f[X]) \le \we(Y) = \osthree(bX) < \widehat{c}(bX) = \widehat{c}(X).$$
\end{proof}

The final result of this section shows that the proviso "there are no $\widehat{c}(X)$-Suslin trees"
in the previous results is essential.

\begin{theorem}\label{tm:sus}
If $L$ is a LOTS then $\ostwo(L) < \widehat{c}(L)$ implies $\den(L) < \widehat{c}(L)$.
Consequently, if $L$ is a $\kappa$-Suslin line for an uncountable regular cardinal $\kappa$
then any NWC image $Y$ of $L$ satisfies $\we(Y) \ge \kappa = \widehat{c}(L)$.
\end{theorem}

\begin{proof}
The family $\mc I$ of all non-empty open intervals of $L$ forms a base for the topology of $L$, hence
it follows from Theorem \ref{tm:t2cutsfrombase} and $\ostwo(L) < \widehat{c}(L)$ that there is
$\mc J \subs \mc I$ with $|\mc J| < \widehat{c}(L)$ such that $\mc J$ $\,T_2$-separates $\tau^+(L)$.
Let $A$ be the set of all end points of the members of $\mc J$, then $|A| < \widehat{c}(L)$ and we
claim that $A$ is dense in $L$. Indeed, if we had $A \cap I = \emptyset$ for some $I \in \mc I$ then
$I \cap J \ne \emptyset$ for any $J \in \mc J$ would imply $I \subs J$, contradicting
our assumption that $\mc J$ $\,T_2$-separates $\tau^+(L)$.

Now, recall that $L$ is a $\kappa$-Suslin line simply means that
$\kappa = \widehat{c}(L) \le \den(L)$. Hence, in this case  we have
$\ostwo(L) = \widehat{c}(L)$ and so,  by Theorem \ref{lm:os2mwcU}, if $Y$ is any NWC image of $L$
then  $\we(Y) \ge \widehat{c}(L)$.
\end{proof}

\section{Pseudo-open images}

The first result of this section, similarly to Theorem \ref{tm:Tych}, yields
an upper bound for the minimum weight of a Tychonov PO image of a Tychonov space $X$
in terms of $\widehat{c}(X)$. However, as being PO is more restrictive than being NWC,
it is not surprising that the upper bound for PO images is larger than the upper bound for NWC images. As we shall see
later, at least consistently, this new upper bound is sharp.

\begin{theorem}\label{tm:po}
Any Tychonov space $X$ has a Tychonov PO image
$Y$ of weight $\we(Y) \le 2^{<\widehat{c}(X)}$.
\end{theorem}

\begin{proof}
We may assume that $X$ is a subspace of some Tychonov cube $[0,1]^\kappa$ and will
show that there is $J \subs \kappa$ with $|J| \le 2^{< \widehat{c}(X)}$ such that the projection map
$\pi_J : [0,1]^\kappa \to [0,1]^J$ restricted to $X$ is PO.
To this end, we first take an elementary submodel $M$ of $H(\lambda)$ for a large enough regular cardinal $\lambda$ such that 
$\,X,\, \kappa \in M$, $\,|M| = 2^{<\widehat c(X)}$, and $\br M;{\mu};\subs M$ for all cardinals ${\mu}<\widehat c(X)$.
This is possible because $\widehat c(X)$ is regular. We then put  $J=M\cap {\kappa}$
and claim that $\pi_J\restriction X : X \to Y = \pi_J[X]$ is PO.
By Theorem \ref{tm:po-equivalence}, this amounts to showing that if $G$ is dense open in $Y$
then $X \cap \pi_J^{-1}[G] = (\pi_J\restriction X)^{-1}[G]$ is dense in $X$.

To see this, let us fix an open base $\mc B \in M$ of $[0,1]$ and let $\mc E$ be the
the family of functions with domain a finite subset of $\kappa$ and range included in $\mc B$.
Clearly, $\mc E \in M$ and $M \vDash \varepsilon \in \mc E$ iff $\varepsilon \in M \cap \mc E$
iff $\varepsilon \in \mc E$ and $\dom(\varepsilon) \subs J$.
As in the first proof of \ref{tm:Tych}, we denote by $[\varepsilon]$ the elementary open set in $[0,1]^\kappa$
determined by $\varepsilon$. Also, for any $\varepsilon \in M \cap \mc E$ we use $[\varepsilon]_J$ to denote
the projection $\pi_J \big[[\varepsilon] \big]$, which is an elementary open set in $[0,1]^J$.

Since $\widehat{c}(Y) \le \widehat{c}(X)$, we can chose a collection $\mc F \subs M \cap \mc E$ with $|\mc F| < \widehat{c}(X)$
such that for each $\varepsilon \in \mc F$ we have $\emptyset \ne Y \cap [\varepsilon]_J \subs G$, moreover the family
$\{Y \cap [\varepsilon]_J : \varepsilon \in \mc F\}$ is disjoint and its union is dense in $G$ and hence in $Y$.
Note that $|\mc F| < \widehat{c}(X)$ implies $\mc F \in M$.

Now, the fact that $H = \bigcup \{Y \cap [\varepsilon]_J : \varepsilon \in \mc F\}$ is dense in $Y$ can be reformulated
as follows: For every $\varepsilon \in M \cap \mc E$, if $Y \cap [\varepsilon]_J \ne \emptyset$ then there is some
$\eta \in \mc F$ such that $Y \cap [\varepsilon]_J \cap [\eta]_J \ne \emptyset$. But $Y \cap [\varepsilon]_J \cap [\eta]_J \ne \emptyset$
is clearly equivalent with $X \cap [\varepsilon] \cap [\eta] \ne \emptyset$, consequently the following
statement is satisfied in $M$:
$$\forall\, \varepsilon \in \mc E\,\exists\,\eta \in \mc F\,\big( X \cap [\varepsilon] \ne \emptyset\,\Rightarrow\, X \cap [\varepsilon] \cap [\eta] \ne \emptyset \big).$$
Since all the three parameters of this formula, namely $\mc E,\, \mc F$, and $X$ belong to $M$, by elementarity it is actually true.
But this just means that $X \cap \pi_J^{-1}[H]$ is dense in $X$, hence so is the larger set $X \cap \pi_J^{-1}[G] $.
\end{proof}

Now we turn to giving the promised results that yield the sharpness of Theorem \ref{tm:po}, at least consistently.
First we need a definition.

\begin{definition}
For a space $X$ we call a {\em funnel in $X$} any decreasing $\omega$-sequence of open sets
that has nowhere dense intersection. We say that $X$ has the {\em small transversal property (STP)}
if for any sequence $\langle \mc F_n : n < \omega\rangle$ of funnels in $X$ with $\mc F_n = \langle U_{n,k} : k < \omega\rangle$
for $n < \omega$, there is a function $g : \omega \to \omega$ such that the "transversal" set $\cup_{n < \omega} U_{n,g(n)}$
is {\em not dense} in $X$.
\end{definition}

Next we need a lemma that has nothing to do with PO, or actually any maps.

\begin{lemma}\label{lm:cent}
If $Y$ is a CCC space and $\mc U$ is any {\em maximal centered} subfamily of $\tau^+(Y)$ then
there is a countable subset $\mc V \subs \mc U$ such that $\cap \mc V$ is nowhere dense in $Y$.
\end{lemma}

\begin{proof}
Let $\mc W$ be a maximal disjoint subset of $\tau^+(Y) \setm \mc U$, then $\mc W$ is countable
because  $Y$ is  CCC. Note that $\mc U$ is closed under finite intersections being maximal centered,
hence for each $W \in \mc W$ there is some $U_W \in \mc U$ with $W \cap U_W = \emptyset$.

We claim that $\cup \mc W$ is dense in $Y$. Indeed, otherwise we would have two disjoint non-empty open subsets
$U_0,\,U_1$ of $Y \setm \overline{\cup \mc W}$ and by the maximality of $\mc W$ both of them
would have to belong to $\mc U$,  that is clearly absurd. But this means that $\mc V = \{U_W : W \in \mc W\}$
is as required.
\end{proof}

This leads us to the following result that yields a necessary condition for PO images of spaces with the STP.

\begin{theorem}\label{tm:pobig}
If the space $X$ has the STP then $\tau^+(Y)$ is not $\sigma$-centered for any PO image $Y$ of $X$.
\end{theorem}

\begin{proof}
Assume, on the contrary, that $f$ is a PO map of $X$ onto $Y$ and $\tau^+(Y)$ is $\sigma$-centered.
Then we may write $\tau^+(Y) = \cup_{n < \omega} \mc U_n$, where we may assume without any loss of generality
that each $\mc U_n$ is maximal centered in $\tau^+(Y)$.

Clearly, if $\tau^+(Y)$ is $\sigma$-centered then $Y$ is CCC, hence we may apply Lemma \ref{lm:cent} to find for each $n < \omega$
a countable subset $\mc V_n \subs \mc U_n$ such that $S_n = \cap \mc V_n$ is nowhere dense in $Y$.
Since $\mc U_n$ is closed under finite intersections, we may clearly assume that $\mc V_n = \{V_{n,k} : k < \omega\}$
is decreasing in $k$ for each $n < \omega$.
Then each $\mc F_n = \< f^{-1}[V_{n,k}] : k < \omega \>$ is a funnel in $X$ because $\cap \mc F_n = f^{-1}[S_n]$
is nowhere dense in $X$, for $f$ is PO.

By the STP then there is a function $g : \omega \to \omega$ such that $\cup_{n < \omega}f^{-1}[V_{n,g(n)}]$
is  not dense in $X$. On the other hand, no matter how we choose $U_n \in \mc U_n$ for each $n < \omega$,
then $\cup_{n < \omega} U_n$ is dense in $Y$ because every $U \in \tau^+(Y)$ belongs to some $\mc U_n$ and
hence intersects $U_n$. This, however, implies that while $W = \cup_{n < \omega}V_{n,g(n)}$ is dense open
in $Y$, its inverse image $f^{-1}[W] = \cup_{n < \omega}f^{-1}[V_{n,g(n)}]$ is not dense in $X$,
contradicting that $f$ is PO.
\end{proof}

\begin{corollary}\label{co:posharp}
Let $X$ be any space with the STP and assume that $Y$ is a PO image of $X$. Then\\
(i) $\den(Y) > \omega$ and\\
\smallskip
(ii) if $Y$ is also CCC and Martin's axiom holds then $\pi(Y) \ge \mathfrak{c}$.
\end{corollary}

\begin{proof}
(i) is obvious because the topology of a separable space is $\sigma$-centered. (ii) follows from
the well-known fact that under Martin's axiom $\pi(Y) < \mathfrak{c}$ implies that $\tau^+(Y)$ is $\sigma$-centered
for any CCC space $Y$, see e.g. \cite{HaJu}.
\end{proof}

Of course, to apply these results we need to give examples of spaces with the STP and that's what we shall do now.

\begin{theorem}\label{tm:stp}
(1) Any standard (i.e. $\omega_1$-) Suslin line $L$ has the STP.\\
(2) Assume that the space $X$ admits a probability measure $\mu$ such that
\begin{enumerate}[(i)]
\item $\mu(U) > 0$ for all $U \in \tau^+(X)$;
\smallskip
\item $\mu(G) = 1$ whenever $G$ is dense open in $X$.
\end{enumerate}
Then $X$ has the STP.
\end{theorem}

\begin{proof}
For (1), consider any sequence $\langle \mc F_n : n < \omega\rangle$ of funnels in $L$ with $\mc F_n = \langle U_{n,k} : k < \omega\rangle$
for $n < \omega$. Now, every open set $U_{n,k}$ is the union of a countable collection $\mc I_{n,k}$ of open intervals because
$L$ is hereditarily Lindelöf. Let $A_{n,k}$ be the set of end points of the intervals in $\mc I_{n,k}$ and let $A$ be the union
of all the  $A_{n,k}$'s. Then $A$ is countable, hence not dense in $L$, so there is a non-empty open interval $J$ of $L$ with $J \cap A = \emptyset$.
But then, for any $n,\,k < \omega$ and $I \in \mc I_{n,k}$, we must have either $J \subs I$ or  $J \cap I = \emptyset$.
Since $\cap \mc F_n$ is nowhere dense, this means that for each $n < \omega$ there is $g(n) < \omega$ such that $J \cap I_{n,g(n)} = \emptyset$,
hence $\cup_{n < \omega} U_{n,g(n)}$ is {\em not dense} in $L$.

\smallskip

To see (2), note first that for any funnel $\mc F = \langle U_{k} : k < \omega\rangle$ in $X$ we have $\mu(\cap\mc F) = 0$ by condition (ii),
consequently the positive values $\mu(U_{k})$ converge to $0$. Thus if $\langle \mc F_n : n < \omega\rangle$ is any sequence of funnels in $X$ then we can
clearly pick  members $U_n$ of $\mc F_n$ for $n < \omega$ such that $\sum_{n < \omega}\mu(U_n) < 1$. But then  $\mu(\cup_{n < \omega}U_n) < 1$ as well, hence
$\cup_{n < \omega}U_n$ is not dense, again by (ii).
\end{proof}

Of course, we already know that a Suslin line does not even have a NWC image of countable weight, hence part (1) of Theorem \ref{tm:stp} does not give us
anything new in that respect. However, we have two interesting examples of type (2).

\begin{example}
Our first example is the interval $[0,1]$ equipped with the standard Lebesgue measure $\lambda$ and not with the standard topology but with the density topology $\delta$;
it is known that this space is Tychonov, see e.g. \cite{T}. Obviously,
the identity map of $[0,1]$ considered as a map from $\delta$ to the standard topology is NWC. On the other hand, by Theorem \ref{tm:stp}
and Corollary \ref{co:posharp}, $[0,1]$ equipped with $\delta$ has no separable PO image. 

This space $X$ is also CCC and hence so is any continuous image of it,
hence, if  Martin's axiom holds then by Corollary \ref{co:posharp} any PO image $Y$ of $X$ satisfies $\we(Y) \ge \pi(Y) \ge \mathfrak{c}$.
Since Martin's axiom is consistent with the continuum $\mathfrak{c}$ being arbitrarily large,
this indeed shows the consistent sharpness of Theorem \ref{tm:po} for CCC Tychonov spaces, i.e. those for which $\widehat c(X) = \omega_1$ holds.
\end{example}

We do not know the answer to the following question.

\begin{problem}
Is there in ZFC a CCC Tychonov space $X$ such that for any PO image $Y$ of $X$ we have $\we(Y) \ge \mathfrak{c}$?
\end{problem}

\begin{example}
Our second example is the compact $L$-space $K$ that was constructed from CH by Kunen in \cite{Ku}. $K$ also carries a measure $\mu$ as in (2) of Theorem \ref{tm:stp},
hence it does not admit a separable PO image. On the other hand, $K$ was constructed as a closed subspace of the Cantor cube $D(2)^{\omega_1}$ in such a way that that $\pi_\omega[K]$,
i.e. the projection of $K$ to the first $\omega$ indices is all of $D(2)^{\omega}$, hence the Cantor set is an NWC image of $K$.
\end{example}

\bigskip

\section{The case of ${\omega}^*$}

In this section we collected everything we could prove concerning the previously discussed topics
in the case of the 0-dimensional compact space $\omega^*$, the \v Cech - Stone remainder of $\omega$.
We think it is interesting that the values on $\omega^*$ of the various cardinal functions we introduced above
coincide with various well-known and well-studied cardinal characteristics of the continuum,
see e.g. Chapter 9 of \cite{Hal}. The topological facts about ${\omega}^*$ that we shall use are well known,
they can be found e.g. in \cite{vM}.

For any infinite set $A \in [\omega]^\omega$ we write $A^* = \overline{A} \setm A$, where the closure is taken
in $\beta\omega$. Then $\mathcal{B} = \{A^* : A \in [\omega]^\omega\}$ is the family of all clopen subsets
of ${\omega}^*$ and hence is a base for ${\omega}^*$.

\begin{theorem}
$\sh(\omega^*) =\ostwo({\omega}^*)= \mathfrak{h}$, where $\mathfrak{h}$ is the shattering number for $[\omega]^\omega$,
see \cite{BPS} or \cite{Hal}.
\end{theorem}

\begin{proof}
$\sh(\omega^*) = \mathfrak{h}$ simply follows by inspecting and comparing the respective definitions,
using that for any $A,\,B \in [\omega]^\omega$ we have $|A \cap B| < \omega$ iff $A^* \cap B^* = \emptyset$.
In fact, this observation motivated us in our choice of terminology for the cardinal function $sh(X)$.

Since $\ostwo(X) \ge \sh(X)$ holds for all $X$, 
it suffices to show $\ostwo({\omega}^*) \le  \mathfrak{h}$. But according to Theorem \ref{tm:log}
we also have $\ostwo(X) \le \sh(X) \cdot \log( \cel(X))$ and thus $$\ostwo({\omega}^*) \le  \mathfrak{h} \cdot \log \mathfrak{c} = \mathfrak{h}$$
because $\cel({\omega}^*) = \mathfrak{c}$ and $\log \mathfrak{c} = \omega$.
\end{proof}

\begin{theorem}
$\osthree({\omega}^*)= \mathfrak{s}$, where $\mathfrak{s}$ is the splitting number for $[\omega]^\omega$.
\end{theorem}
\begin{proof}
If $\mc S\subs \br {\omega};{\omega};$ is a splitting family for $[\omega]^\omega$ then
 \begin{displaymath}
  \{S^*, {\omega}^*\setm S^*: S\in \mc S\}
 \end{displaymath}
clearly $T_3$-splits $\mc B$, a base for ${\omega}^*$, and so it  $T_3$-splits $\tau^+({\omega}^*)$ as well.
Thus we have $\osthree({\omega}^*) \le \mathfrak{s}$.

Next, by Theorem \ref{tm:t3cutsfrombase},
there is a $T_3$-splitting family $\mc U$ for ${\omega}^*$ of size $\osthree({\omega}^*)$ with $\mc U \subs \mc B$.
Now, it is straight forward to check that if $\mc U = \{A^* : A \in \mc S\}$ then $\mc S$
is a splitting family for $[\omega]^\omega$, hence $\osthree({\omega}^*) \ge \mathfrak{s}$.
\end{proof}

\begin{theorem}
 $\costhree({\omega}^*)= \mathfrak{c}$.
\end{theorem}

\begin{proof}
Now $\costhree({\omega}^*) \le \we({\omega}^*) = \mathfrak{c}$ is trivial, hence by Theorem \ref{tm:t3cutsfrombase}
it suffices to show that if $\mc U \subs \mc B$ has cardinality $|\mc U| < \mathfrak{c}$ then $\mc U$ does not $T_3$-split
$CR({\omega}^*)$.

To see this, we apply Pospi\v sil's celebrated result from \cite{Po} to pick a point $p \in {\omega}^*$
of character $\chi(p,{\omega}^*) = \mathfrak{c}$ and put $$\mc V = \{U \in \mc U : p \in U\} \cup \{{\omega}^* \setm U : U \in \mc U \text{ and } p \notin U\}.$$
Let us put $F = \cap \mc V$, then $p \in F$ and $F$ is an infinite closed subset of ${\omega}^*$ because, by compactness, we have
$\psi(p,{\omega}^*) = \chi(p,{\omega}^*) = \mathfrak{c}$. But it is well-known that any infinite closed subset of ${\omega}^*$
includes a copy of ${\omega}^*$ that is crowded, while no two members of $\mc U$ split even $F$.
 \end{proof}

The last two theorems together with \ref{tm:os2nwcnormal} and \ref{tm:cos2nwcnormal} yield the following
corollary determining the minimum weight of NWC, resp. CP images of $\omega^*$.

\begin{corollary}
\begin{enumerate}[(i)]
  \item The minimum weight of a NWC image of $\omega^*$ is equal to $\mathfrak{s}$.
  \item  The weight of any CP image of $\omega^*$ is equal to $\mathfrak{c}$.
\end{enumerate}
\end{corollary}

We can also determine the weight of PO images of $\omega^*$.

\begin{theorem}
For any PO image $Y$ of $\omega^*$ we have $\widehat{c}(Y) = \mathfrak{c}^+$, hence $\we(Y) = \mathfrak{c}$.
\end{theorem}

\begin{proof}
Assume that $f$ is a PO map of $\omega^*$ onto $Y$.
The standard Cantor tree argument gives a disjoint family $\mathcal{G}$ of non-empty closed $G_\delta$ sets in $Y$
with $|\mathcal{G}| = \mathfrak{c}$. But for every $G \in \mc G$ then $f^{-1}(G)$ is a non-empty closed $G_\delta$ set in $\omega^*$
and thus has non-empty interior, say $U_G$. But then, as $f$ is PO, by proposition \ref{tm:po-equivalence} we have
$\emptyset \ne int \big(f[\overline{U_G}]\big) \subs G$, and we are done.
\end{proof}

The following problem contains the only thing that we do not know about $\omega^*$.

\begin{problem}
What is $\costwo({\omega}^*)$?
\end{problem}

\bigskip

\section{An application to densely $k$-separable spaces}

We start this section with a couple of simple definitions taken from \cite{DJ}.

\begin{definition}
A space $X$ is called {\em $k$-separable} if it has a $\sigma$-compact dense subset.
We say that $X$ is {\em densely $k$-separable} if  every dense subspace of $X$ is $k$-separable.
\end{definition}

It was shown in \cite{DJ} that every densely $k$-separable compact space is actually densely separable,
or equivalently, has countable $\pi$-weight. The aim of this section is to present a result on
NWC images of densely $k$-separable spaces which provides an alternative to --
the lengthily and tediously proved -- Lemma 3.1 of \cite{DJ} that was crucial
in the proof of the main result of \cite{DJ}.

We recall that a space $X$ is called {\em feebly compact} if every locally finite family of open sets in $X$
is finite. This is clearly equivalent with the condition that for every decreasing sequence
$\{U_n : n < \omega\} \subs \tau^+(X)$ we have $\cap \{\overline{U_n} : n < \omega\} \ne \emptyset$.
A space is pseudocompact iff it is a feebly compact Tychonov space.

\begin{theorem}
If the feebly compact $\pi$-regular space $X$ is densely $k$-separable then $\sh(X) = \omega$.
Consequently, any pseudocompact space has a compact metrizable NWC image.
\end{theorem}

\begin{proof}
We shall actually prove the contrapositive of our statement: If $X$ is feebly compact and $\pi$-regular
with $\sh(X) > \omega$ then $X$ is not densely $k$-separable. Since densely $k$-separable spaces are trivially
CCC, we are done if $X$ is not CCC. So we may assume that $X$ is CCC (i.e. $\widehat{c}(X) = \omega_1$) and hence, 
by Theorem \ref{tm:sh_le}, $\sh(X) = \omega_1$.

We may also assume, without any loss of generality, that  $\sh(U) = \omega_1$ holds for all $U \in \tau^+(X)$.
Indeed, let $\mc U \subs \tau^+(X)$ be a maximal disjoint collection with $\sh(U) = \omega$ for each $U \in \mc U$. 
Then $\bigcup \mc U$ cannot be dense in $X$ because that clearly would imply $\sh(X) = \omega$ as well.
Thus we have a regular closed subset $W$ of $X$ with $\bigcup \mc U \cap W = \emptyset$. But $W$ is also
feebly compact and $\pi$-regular, and clearly, every $V \in \tau^+(W)$ satisfies $\sh(V) = \omega_1$.
Since regular closed subsets of densely $k$-separable are again densely $k$-separable, we may simply replace $X$ with $W$.

Let us now consider the shattering tree $\mc T$ of height $\omega_1$ for $X$ that we constructed in the proof 
of Theorem \ref{tm:sh_le}. (We are going to use the notation and terminology given there.) The only difference in
the construction of $\mc T$ is that the immediate successors in $\mc T_{\alpha+1}$ of any $R \in \mc T_\alpha$ are not $P(R)$ and $Q(R)$ but
the members of a maximal collection $\mc U(R) \subs \RO^+(X)$ such that for every  $U \in \mc U(R)$
we have $\overline{U} \subs R$, moreover $\overline{U} \cap \overline{V} = \emptyset$ for distinct $U,\,V \in \mc U(R)$.
Using the $\pi$-regularity of $X$, it is easy to check that everything we did concerning the shattering tree $\mc T$ in the proof
of Theorem \ref{tm:sh_le} remains now valid. One additional consequence in the present case is that for every limit ordinal $\alpha < \omega_1$
and $b \in B_\alpha$ we have $$S_b = \bigcap_{\beta < \alpha}\overline{b(\beta)} = \bigcap_{\beta < \alpha}b(\beta).$$
In particular, this means that each $S_b$ is a $G_\delta$-set.

First we show that
for each $\alpha < \omega_1$ the union of the level $\mc T_\alpha$ of $\mc T$ is dense in $X$. Indeed, given any $U \in \tau^+(X)$, the assumption
$\sh(U) = \omega_1$ implies that there is $V \in \tau^+(X)$ such that for all $\beta \le \alpha$ no two members of $\mc T_\beta$ meets $V$.
But then by the statement ($I_\alpha$) we proved there we have some $R \in \mc T_\alpha$ such that $V \subs R$,
consequently $U \cap R \ne \emptyset$.

Let us denote by $Br$ the set of all maximal branches of $\mc T$. Clearly, $b \in Br$ means that for some limit $\alpha < \omega_1$
we have $b \in B_\alpha$ and $S_b$ has empty interior.
We claim next that $S = \bigcup \{S_b : b \in Br\}$ is a dense subset of $X$.
To see this, by $\pi$-regularity, it suffices to show that for every $U \in \tau^+(X)$
there is $b \in Br$ with $\overline{U} \cap S_b \ne \emptyset$. Assume, on the contrary, that for every $b \in Br$
we have $\overline{U} \cap S_b = \emptyset$. This implies that for every $b \in Br$ with $b \in B_\alpha$ there is
some $\beta < \alpha$ such that $U \cap b(\beta) = \emptyset$. Indeed, otherwise we had by feeble compactness of $X$ that
$$\emptyset \ne \bigcap_{\beta < \alpha}\overline{U \cap b(\beta)} \subs \overline{U} \cap S_b.$$

Now, for every $\alpha < \omega_1$ we may choose  $R_\alpha \in \mc T_\alpha$ such that $U \cap R_\alpha \ne \emptyset$
and $b_\alpha \in Br$ with $R_\alpha = b_\alpha(\alpha)$. But by our indirect assumption we have $U \cap S_{b_\alpha} = \emptyset$, 
so there is a smallest ordinal $\beta_\alpha > \alpha$ with $U \cap b_\alpha(\beta_\alpha) = \emptyset$. 
Then we may clearly find an uncountable set of ordinals $L \subs \omega_1$ such that $\alpha, \delta \in L$ and $\alpha < \delta$ imply
$\beta_\alpha < \beta_\delta$. However, then we have $U \cap b_\delta(\beta_\alpha) \ne \emptyset$ while $U \cap b_\alpha(\beta_\alpha) = \emptyset$,
hence $b_\delta(\beta_\alpha) \cap b_\alpha(\beta_\alpha) = \emptyset$, consequently $b_\delta(\beta_\delta) \cap b_\alpha(\beta_\alpha) = \emptyset$
as well. But this would mean that $\{b_\alpha(\beta_\alpha) : \alpha \in L\}$ is an uncountable antichain in $\mc T$, 
a contradiction implying that $S$ is indeed dense in $X$. 

Let us now fix an enumeration $\{B_\xi : \xi < \mu =|Br|\}$ of the set of branches $Br$ and for each $\xi < \mu$ define
$Q_\xi = S_{b_\xi} \setm \overline{\bigcup_{\eta < \xi}S_{b_\eta}}$. Then each $Q_\xi$ is again a $G_\delta$-set and
the family  $\{Q_\xi : \xi < \mu\}$ is left-separated in the sense of section 3 of \cite{DM}. 

We claim that $Z = \bigcup \{Q_\xi : \xi < \mu\} \subs Y$ is also dense in $X$. Indeed, for any $U \in \tau^+(X)$ there is
a smallest ordinal $\xi$ such that $U \cap S_{b_\xi} \ne \emptyset$, which implies that $U \cap \overline{\bigcup_{\eta < \xi}S_{b_\eta}} = \emptyset$,
hence $U \cap Q_\xi \ne \emptyset$. But then, by Theorem 3.1 of \cite{DM}, for every compact subset $K \subs Q$ there is a countable set
of indices $a_K \subs \mu$ such that $K \subs \bigcup_{\xi \in a_K} Q_\xi$, and then this also holds for any $\sigma$-compact $K \subs Q$.
However, for any countable set of branches $A \subs Br$ the union of $\{S_{b} : b \in A\}$ is not dense in $X$.
Indeed, let $b \in B_{\alpha_b}$ for $b \in A$ and choose the limit $\alpha < \omega_1$ above $\zeta = \sup \{\alpha_b : b \in A\}$. 
Then for any branch $d \in B_\alpha \cap Br$ we have $d(\zeta + 1) \cap \bigcup \{S_b : b \in A\} = \emptyset$.
Thus we may conclude that the dense subset $Z$ of $X$ has no $\sigma$-compact dense subset, hence $X$ is not
densely $k$-separable.

Now, if $X$ is pseudocompact, hence Tychonov, and densely $k$-separable then, by Theorem \ref{tm:sh_le}, $\sh(X) = \omega$ implies
that $X$ has a Tychonov NWC image $Y$ of countable weight, hence $Y$ is metrizable. But clearly, $Y$ is also pseudocompact, and
then compact as well. 
\end{proof}

\begin{corollary}\label{co:pscpt}
Let X be a crowded pseudocompact space, then there is a partition $\mc Z$ of $X$ consisting of nowhere dense closed $G_\delta$ sets and satisfying 
that, for all non-empty regular closed subsets $R$ of $X$, the set $\{Z \in \mc Z : R \cap Z \ne \emptyset\}$ has cardinality $\mathfrak{c}$.
\end{corollary}

\begin{proof}
Let $f : X \to Y$ be an NWC surjection of $X$ onto the compact metrizable space $Y$. Then $\mc Z = \{f^{-1}(y) : y \in Y\}$ is the required partition of $X$.
The last requirement is immediate from the fact that $f[R]$ is a crowded (pseudo)compact set in $Y$, for $f$ is NWC. 
\end{proof}

It is easy to check that this Corollary  could replace Lemma 3.1 in the proof of
the main result of \cite{DJ}.


\bigskip


\begin{thebibliography}{12}

\bibitem{BPS}   Balcar, B; Pelant, J; Simon, P;
The space of ultrafilters on N covered by nowhere dense sets.
Fund. Math. 110 (1980), no. 1, 11–24.

\bibitem{DJ} Dow A and Juhász I, Densely $k$-separable compacta are
densely separable, Top. Appl., to appear,
https://arxiv.org/pdf/1810.05071.pdf

\bibitem{DM} Dow A and Moore J, Tightness in $\sigma$-compact spaces,
 Topology Proc. 46 (2015), 213–232. 

\bibitem{HaJu} Hajnal A and Juhász I, A consequence of Martin’s axiom,
Indag Math., 33 (1971), pp. 457–463.

\bibitem{Hal}  Halbeisen L J; Combinatorial Set Theory, With a gentle introduction to forcing.
Second edition, Springer Monographs in Mathematics. Springer, Cham, 2017.

\bibitem{Ju} Juhász I, Cardinal functions – ten years later, Math. Centre
Tract 123 (1980). Amsterdam.

\bibitem{JSSz} Juhász, I; Soukup, L; Szentmiklóssy, Z; Resolvability of spaces having small spread or extent, 
Topology Appl. 154 (2007), no. 1, 144–154.

\bibitem{Ku} Kunen K.  A compact L-space under CH. Topology Appl. 12 (1981), no. 3, 283–287.

\bibitem{Po} Pospi\v sil B,  On bicompact spaces. Publ. Fac. Sci. Univ. Masaryk 270 (1939), 3-16.

\bibitem{T} Tall, F. D. The density topology. Pacific J. Math. 62 (1976), no. 1, 275–284.

\bibitem{vM}  van Mill J, An introduction to $\beta\omega$.
Handbook of set-theoretic topology, 503–567, North-Holland, Amsterdam, 1984.


\end{thebibliography}
\end{document}